\documentclass[11pt]{amsart}  
\usepackage{amsthm, amsmath, amscd, amssymb, latexsym, stmaryrd, color}

\usepackage[all]{xypic}
\input xypic
\usepackage{color}
\usepackage[top=3.0cm, bottom=3.0cm, left=3.0cm, right=3.0cm]{geometry}

\theoremstyle{plain}
\newtheorem*{thm}{Theorem}

\newtheorem{theorem}{Theorem}[section]
\newtheorem{lemma}[theorem]{Lemma}

\newtheorem{prop-def}[theorem]{Proposition-Definition}
\newtheorem{corollary}[theorem]{Corollary}

\newtheorem{prob}{Problem}

\theoremstyle{definition}

\newtheorem{remark}[theorem]{Remark}

\theoremstyle{remark}
\newtheorem*{ack}{Acknowledgement}

\usepackage[mathscr]{eucal}
\usepackage{graphics, graphpap}
\usepackage{array, tabularx, longtable}
\usepackage{color}
\numberwithin{equation}{section}

\def\Var{\mathrm{Var}}

\def\pr{\mathrm{pr}}

\def\sr{\mathrm{sr}}

\def\ord{\mathrm{ord}}
\def\Lbb{\mathbb{L}}

\def\Spec{\mathrm{Spec}}

\def\BM{\mathrm{BM}}

\def\L{\mathbb{L}}

\title[Contact loci, nearby cycles of nondegenerate polynomials]{\bf Geometry of nondegenerate polynomials: Motivic nearby cycles and Cohomology of contact loci}  
\author[Q.T. L\^e]{L\^e Quy Thuong}

\address{University of Science, Vietnam National University, Hanoi, 
\newline 
\indent 334 Nguyen Trai Street, Thanh Xuan District, Hanoi, Vietnam}
\email{leqthuong@gmail.com}

\author[T.T. Nguyen]{Nguyen Tat Thang}
\address{Institute of Mathematics, Vietnam Academy of Science and Technology
\newline
\indent 18 Hoang Quoc Viet Road, Cau Giay District, Hanoi, Vietnam}
\email{ntthang@math.ac.vn}

\thanks{}


\begin{document}           

\begin{abstract}
We study polynomials with complex coefficients which are nondegenerate in two senses, one of Kouchnirenko and the other with respect to its Newton polyhedron, through data on contact loci and motivic nearby cycles. Introducing an explicit description of these quantities we can answer in part to questions concerning the motivic nearby cycles of restriction functions and the integral identity conjecture in the context of Newton nondegenerate polynomials. Furthermore, in the nondegeneracy in the sense of Kouchnirenko, we give calculations on cohomology groups of the contact loci.
\end{abstract}
\maketitle                 

\section{Introduction}\label{sec1}
Let $f$ be a nondegenerate $\mathbb C$-polynomial in the sense of Kouchnirenko (cf. Section \ref{nondegenerate}) vanishing at the origin $O$ of $\mathbb C^d$. The problem of computing the motivic Milnor fiber $\mathscr S_{f,O}$ in terms of the Newton polyhedron $\Gamma$ of $f$ was firstly mentioned by Guibert in 2002 (cf. \cite{G}). Recently, Steenbrink and Bultot-Nicaise obtain solutions in terms of toric geometry (\cite{St2}), or of log smooth models (\cite{BN}). Their formula for $\mathscr{S}_{f,O}$ allows to compute the Hodge spectrum of the singularity of $f$ at $O$ by means of the additivity of the Hodge spectrum operator. In this article, we will show that the formula also provides a way to explore the following problem for Newton nondegenerate polynomials.

\begin{prob}\label{pro1}
Let $f$ be in $\mathbb C[x_1,\dots,x_d]$ with $f(O)=0$, and let $H$ be a hyperplane in $\mathbb C^d$. What is the relation between $\mathscr{S}_{f,O}$ and $\mathscr{S}_{f|_H,O}$?
\end{prob}

The question concerns a motivic analogue of a monodromy relation of a complex singularity and its restriction to a generic hyperplane studied early in \cite{LDT}. 
For $n\in \mathbb N^*$, the $n$-iterated contact locus $\mathscr X_{n,O}(f)$ admits a decomposition into $\mu_n$-invariant $\mathbb C$-subvarieties $\mathscr{X}_{J,a}^{(n)}$ along $a\in \mathbb N_{>0}^J$ and $J\subseteq [d]:=\{1,\dots,d\}$. 
The nondegeneracy of $f$ allows to describe $\mathscr{X}_{J,a}^{(n)}$ via $\Gamma$, as in Theorem \ref{mainthm1}, which is the key step to compute the motivic zeta function $Z_{f,O}(T)$ and the motivic Milnor fiber $\mathscr S_{f,O}$, which yields Theorem \ref{thm4.1}. For every face $\gamma$ of $\Gamma$, let $J_{\gamma}$ be the unique subset of $[d]$ such that $\gamma$ is contained in the hyperplanes $x_j=0$ for all $j\not\in J_{\gamma}$ and not contained in the other coordinate hyperplanes, and let $X_{\gamma}(0)$ (resp. $X_{\gamma}(1)$) be the $\mathbb C$-subvariety of $\mathbb G_{m,\mathbb C}^d$ defined by the face function $f_{\gamma}$ (resp. $f_{\gamma}-1$). Let $K$ be the set of all compact faces of $\Gamma$. 

\begin{thm}[see Theorem \ref{thm4.1}]
Let $f$ be nondegenerate in the sense of Kouchnirenko such that $f(O)=0$. Then the identity  $\mathscr{S}_{f,O}=\sum_{\gamma\in K}(-1)^{|J_{\gamma}|+1-\dim(\gamma)}\left([X_{\gamma}(1)]-[X_{\gamma}(0)]\right)$, holds true in the monodromic Grothendieck ring of $\mathbb C$-varieties endowed with $\hat\mu$-action.
\end{thm}

We choose the hyperplane defined by $x_d=0$ to be $H$ in Problem \ref{pro1}, and consider for any $n\geq m$ in $\mathbb N^*$ the so-called $(n,m)$-iterated contact locus $\mathscr{X}_{n,m,O}(f,x_d)$ of the pair $(f,x_d)$. It is a $\mu_n$-invariant $\mathbb C$-subvariety of $\mathscr X_{n,O}(f)$. Then we show in this article that the formal series   
$$
Z_{f,x_d,O}^{\Delta}(T):=\sum_{n\geq m\geq 1}\left[\mathscr{X}_{n,m,O}(f,x_d)\right]\L^{-(n+m)d}T^n
$$
is rational and it can be described via data of $\Gamma$. We put $\mathscr S_{f,x_d,O}^{\Delta}:=-\lim_{T\to \infty}Z_{f,x_d,O}^{\Delta}(T)$. Using the description of $\mathscr S_{f,x_d,O}^{\Delta}$ together with Theorem \ref{thm4.1}, a solution to Problem \ref{pro1} for the nondegeneracy in the sense of Kouchnirenko can be realized as in the following theorem.

\begin{thm}[part of Theorem \ref{restricting}] 
With $f$ as previous, the identity $\mathscr{S}_{f,O}=\mathscr{S}_{f|_H,O}+\mathscr S_{f,x_d,O}^{\Delta}$ holds in the monodromic Grothendieck ring of $\mathbb C$-varieties with $\hat\mu$-action.
\end{thm}

We also obtain a similar result on the motivic nearby cycles (Theorem \ref{restricting}). An important consequence of Theorems \ref{mainthm1}, \ref{thm4.1} and \ref{restricting} is a very elementary proof of the integral identity conjecture for Newton nondegenerate polynomials (Corollary \ref{IIC}). 

According to \cite[Conjecture 1.5]{BFLN}, it is expected that the singular cohomology groups with compact support of the $\mathbb C$-points of the contact loci are nothing but the Floer cohomology groups of the powers of the monodromy of the singularity (cf. \cite{Mc}). Here, we are interested in a smaller problem on cohomology groups of $\mathscr X_{n,O}(f)$ (compare with \cite[Theorem 1.1]{BFLN}).

\begin{prob}\label{pro2}
Let $f$ be a polynomial over $\mathbb C$ vanishing at the origin $O$. Compute the cohomology groups with compact support $H_c^m(\mathscr X_{n,O}(f),\mathbb C)$.
\end{prob}

We devote Section \ref{Sec5-cloci} to study this problem for nondegenerate singularities in the sense of Kouchnirenko not only using sheaf cohomology with compact support but also the Borel-Moore homology $H_*^{\BM}$. Write $\mathscr X_{n,O}(f)=\bigsqcup_{(J,a)\in \widetilde{\mathcal P}_n}\mathscr{X}_{J,a}^{(n)}$, where $\widetilde{\mathcal P}_n$ is defined right after (\ref{otherdecomp}). Let $\eta: \widetilde{\mathcal P}_n\to \mathbb{Z}$ be the function defined by $\eta(J,a)=\dim_{\mathbb{C}}\mathscr{X}_{J,a}^{(n)}$. We prove the following results:
\begin{thm}[Theorems \ref{thm5.4}, \ref{thm5.6}] 
For $f$ as in Problem \ref{pro2} and nondegenerate in the sense of Kouchnirenko, there exist spectral sequences
\begin{gather*}
E_{p,q}^1:=\bigoplus_{\eta(J, a)=p}H^{\BM}_{p+q}(\mathscr{X}_{J,a}^{(n)})\Longrightarrow  H^{\BM}_{p+q}(\mathscr X_{n,O}(f)),\\
E^{p,q}_1:=\bigoplus_{\eta(J,a)=p}H^{p+q}_c(\mathscr{X}_{J,a}^{(n)}, \mathcal{F})\Longrightarrow  H^{p+q}_c(\mathscr X_{n,O}(f), \mathcal{F}),
\end{gather*}
for any sheaf of abelian groups $\mathcal{F}$ on $\mathscr X_{n,O}(f)$. 
\end{thm}

In particular, by applying the second spectral sequence with $ \mathcal{F}$ being a constant sheaf, we obtain a spectral sequence converging to the compact support cohomology groups of contact loci with complex coefficient whose first page is a direct sum of (singular) homology of $X_{\gamma}(0)$ and $X_{\gamma}(1)$ (see Corollary \ref{corconstantsheaf}).


\section{Preliminaries}\label{sec2}
\subsection{Monodromic Grothendieck ring of varieties}
Let $S$ be an algebraic $\mathbb C$-variety. Let $\Var_S$ be the category of $S$-varieties, with objects being morphisms of algebraic $\mathbb C$-varieties $X\to S$ and a morphism in $\Var_S$ from $X\to S$ to $Y\to S$ being a morphism of algebraic $\mathbb C$-varieties $X\to Y$ commuting with $X\to S$ and $Y\to S$. Denote by $\hat{\mu}$ the limit of the projective system $\mu_{nm}\to \mu_n$ given by $x\mapsto x^m$, with $\mu_n=\Spec \mathbb C[\xi]/(\xi^n-1)$ the group scheme over $\mathbb C$ of $n$th roots of unity. An action of $\hat\mu$ on a variety $X$ is an action of a group $\mu_n$ on $X$, and the action is good if every orbit is contained in an affine open subset of $X$. By definition, an action of $\hat\mu$ on an affine Zariski bundle $X\to B$ is affine if it is a lifting of a good action on $B$ and its restriction to all fibers is affine.

The Grothendick group $K_0^{\hat\mu}(\Var_S)$ is defined to be an abelian group generated by symbols $[X\to S]$, $X$ endowed with a good $\hat{\mu}$-action and $X\to S$ in $\Var_S$, such that:
\begin{itemize}
\item[i)] $[X\to S]=[Y\to S]$ if $X$ and $Y$ are $\hat\mu$-equivariant $S$-isomorphic;  
\item[ii)] $[X\to S]=[Y\to S]+[X\setminus Y\to S]$ if $Y$ is a $\hat\mu$-invariant closed subvariety in $X$; and
\item[iii)] $[X\times\mathbb{A}_{\mathbb C}^n,\sigma]=[X\times\mathbb{A}_{\mathbb C}^n,\sigma']$ 
if $\sigma$ and $\sigma'$ are liftings of the same $\hat{\mu}$-action on $X$ to $X\times\mathbb{A}_{\mathbb C}^n$. 
\end{itemize}
There is a natural ring structure on $K_0^{\hat\mu}(\Var_S)$ in which the product is induced by the fiber product over $S$. The unit $1_S$ for the product is the class of the identity morphism $S\to S$ with $S$ endowed with trivial $\hat\mu$-action. Denote by $\Lbb$ (or $\Lbb_S$) the class of the trivial line bundle $S\times\mathbb{A}^1\to S$, and define the localized ring $\mathscr{M}_S^{\hat{\mu}}$ to be $K_0^{\hat\mu}(\Var_S)[\Lbb^{-1}]$.

Let $f:S\to S'$ be a morphism of algebraic $\mathbb C$-variety. Then we have two important morphisms associated to $f$, which are the ring homomorphism $f^*: \mathscr{M}_{S'}^{\hat{\mu}}\to \mathscr{M}_S^{\hat{\mu}}$ induced from the fiber product (the pullback morphism) and the $\mathscr M_{\mathbb C}$-linear homomorphism $f_!: \mathscr{M}_S^{\hat{\mu}}\to \mathscr{M}_{S'}^{\hat{\mu}}$ defined by the composition with $f$ (the push-forward morphism). When $S'$ is $\Spec \mathbb C$, one usually writes $\int_S$ instead of $f_!$.

\subsection{Rational series and limit}\label{2.3}
Let $\mathscr A$ be either $\mathbb{Z}[\Lbb,\Lbb^{-1}]$ or $\mathscr{M}_S^{\hat\mu}$ as a ring. Let $\mathscr A[[T]]_{\sr}$ be the $\mathscr A$-submodule of $\mathscr A[[T]]$ generated by $1$ and by finite products of elements of the form $\frac{\Lbb^aT^b}{1-\Lbb^aT^b}$ with $(a,b)$ in $\mathbb{Z}\times\mathbb{N}_{>0}$. Each element of $\mathscr A[[T]]_{\sr}$ is called a {\it rational series}. By \cite{DL1}, there is a unique $\mathscr A$-linear morphism $\lim\limits_{T\rightarrow\infty}: \mathscr A[[T]]_{\sr}\rightarrow \mathscr A$ which sends $\frac{\Lbb^aT^b}{1-\Lbb^aT^b}$ to $-1$. 

During this article, we denote $[d]:=\{1,\dots,d\}, d\in \mathbb{N}^{*}$. For $J$ contained in $[d]$, we denote by $\mathbb{R}_{\geq 0}^J$ the set of $(a_j)_{j\in J}$ with $a_j$ in $\mathbb{R}_{\geq 0}$ for all $j\in J$, and by $\mathbb{R}_{>0}^J$ the subset of $\mathbb{R}_{\geq 0}^J$ consisting of $(a_j)_{j\in J}$ with $a_j>0$ for all $j\in J$. Similarly, one can define the sets $\mathbb{Z}_{\geq 0}^J$, $\mathbb{Z}_{> 0}^J$ and $\mathbb{N}_{>0}^J$. Let $\sigma$ be a rational polyhedral convex cone in $\mathbb{R}_{>0}^J$ and let $\overline{\sigma}$ denote its closure in $\mathbb{R}_{\geq 0}^J$ with $J$ a finite set. Let $\ell$ and $\ell'$ be two integer linear forms on $\mathbb{Z}^J$ positive on $\overline{\sigma}\setminus\{(0,\dots,0)\}$. Lemma 2.1.5 in \cite{G} tells us that if $\sigma$ is open in its linear span and $\overline{\sigma}$ is generated by part of a $\mathbb{Z}$-basis of the $\mathbb{Z}$-module $\mathbb{Z}^J$, then the series $$S_{\sigma,\ell,\ell'}(T):=\sum_{a\in\sigma\cap\mathbb{N}_{>0}^J}\mathbb{L}^{-\ell'(a)}T^{\ell(a)}$$
is in $\mathbb{Z}[\mathbb{L},\mathbb{L}^{-1}][[T]]_{\sr}$ and $\lim\limits_{T\rightarrow\infty}S_{\sigma,\ell,\ell'}(T)=(-1)^{\dim(\sigma)}$.

\subsection{Motivic nearby cycles of regular functions}\label{remark}
For any $\mathbb C$-variety $X$, let $\mathscr{L}_n(X)$ be the space of $n$-jets on $X$, and $\mathscr{L}(X)$ the arc space on $X$, which is the limit of the projective system of spaces $\mathscr{L}_n(X)$ and canonical morphisms $\mathscr{L}_m(X)\to \mathscr{L}_n(X)$ for $m\geq n$.  The group $\hat\mu$ acts on $\mathscr{L}_n(X)$ via $\mu_n$ in such a natural way that $\xi\cdot\varphi(t)=\varphi(\xi t)$ for $\xi\in \mu_n$. 

From now on, we assume that the $\mathbb C$-variety $X$ is smooth and of pure dimension $d$. Consider a regular function $f:X\to \mathbb{A}_{\mathbb C}^1$, with the zero locus $X_0$. For $n\geq 1$ one defines the {\it $n$-iterated contact locus} of $f$ as follows
$$\mathscr{X}_n(f)=\{\varphi\in\mathscr{L}_n(X)\mid f(\varphi)=t^n\mod t^{n+1}\}.$$
Clearly, this variety is invariant by the $\hat\mu$-action on $\mathscr{L}_n(X)$ and admits a morphism to $X_0$ given by $\varphi(t)\mapsto \varphi(0)$, which defines an element $[\mathscr{X}_n(f)]:=[\mathscr{X}_n(f)\to X_0]$ in $\mathscr{M}_{X_0}^{\hat\mu}$. We consider Denef-Loeser's motivic zeta function
$Z_f(T)=\sum_{n\geq 1}\left[\mathscr{X}_n(f)\right]\Lbb^{-nd}T^n.$
They prove in \cite{DL1} that $Z_f(T)$ is in $\mathscr{M}_{X_0}^{\hat\mu}[[T]]_{\sr}$, and call the limit $\mathscr S_f:=-\lim\limits_{T\to\infty}Z_f(T)$ in $\mathscr{M}_{X_0}^{\hat\mu}$ the {\it motivic nearby cycles} of $f$. If $x$ is a closed point of $X_0$, the $\mathbb C$-variety
$$\mathscr{X}_{n,x}(f)=\{\varphi\in\mathscr{L}_n(X)\mid f(\varphi)=t^n\mod t^{n+1}, \varphi(0)=x\},$$
is also invariant by the $\hat\mu$-action on $\mathscr L_n(X)$, called the {\it $n$-iterated contact locus} of $f$ at $x$. It is also proved that the zeta function 
$Z_{f,x}(T)=\sum_{n\geq 1}\left[\mathscr{X}_{n,x}(f)\right]\Lbb^{-nd}T^n$
is in $\mathscr{M}_{\mathbb C}^{\hat\mu}[[T]]_{\sr}$. The limit $\mathscr S_{f,x}=-\lim\limits_{T\to\infty}Z_{f,x}(T)$ is called the {\it motivic Milnor fiber} of $f$ at $x$. Obviously, if $\iota$ is the inclusion of $\{x\}$ in $X_0$, then $\mathscr S_{f,x}=\iota^*\mathscr S_f$ in $\mathscr{M}_{\mathbb C}^{\hat\mu}$.

We now modify slightly the motivic zeta functions of several functions in \cite{G} and \cite{GLM2}. For a pair of regular functions $(f,g)$ on $X$, we denote by $X_0:=X_0(f,g)$ their common zero locus. For $n\geq m$ in $\mathbb N^*$, we define
$$\mathscr X_{n,m}(f,g):=\left\{\varphi\in\mathscr{L}_n(X)\mid f(\varphi)= t^n\mod t^{n+1}, \ord_tg(\varphi)=m\right\}.$$
We can check that $\mathscr{X}_{n,m}(f,g)$ is invariant under the natural $\mu_n$-action on $\mathscr{L}_n(X)$, and that there is an obvious morphism of $\mathbb C$-varieties $\mathscr{X}_{n,m}(f,g)\to X_0$; from which we obtain the class $[\mathscr{X}_{n,m}(f,g)]$ of that morphism in $\mathscr{M}_{X_0}^{\hat\mu}$. Consider the series
$$
Z_{f,g}^{\Delta}(T):=\sum_{n\geq m\geq 1}\big[\mathscr X_{n,m}(f,g)\big]\Lbb^{-nd}T^n$$
in $\mathscr M_{X_0}^{\hat\mu}[[T]]$. For any closed point $x\in X_0$, we can define $Z_{f,g,x}^{\Delta}(T)$ in $\mathscr M_{\mathbb C}^{\hat\mu}[[T]]$ as above with $\mathscr{X}_{n,m}(f,g)$ replaced by its $\mu_n$-invariant subvariety $\mathscr{X}_{n,m,x}(f,g):=\{\varphi\in \mathscr{X}_{n,m}(f,g)\mid \varphi(0)=x\}$. The rationality of the series $Z_{f,g}^{\Delta}(T)$ and $Z_{f,g,x}^{\Delta}(T)$ are stated in \cite[Th\'eor\`eme 4.1.2]{G} and \cite[Section 2.9]{GLM2}, up to the isomorphism of rings $\mathscr{M}_{X_0}^{\hat{\mu}}\cong\mathscr{M}_{X_0\times\mathbb{G}_m}^{\mathbb{G}_m}$ (see \cite[Proposition 2.6]{GLM1}), where Guibert-Loeser-Merle's result is done in the framework $\mathscr{M}_{X_0\times\mathbb{G}_m}^{\mathbb{G}_m}$. Put $\mathscr S_{f,g}^{\Delta}:=-\lim\limits_{T\to \infty}Z_{f,g}^{\Delta}(T)$ and $\mathscr S_{f,g,x}^{\Delta}:=-\lim\limits_{T\to \infty}Z_{f,g,x}^{\Delta}(T)$.


\section{Motivic nearby cycles of a nondegenerate polynomial and applications}\label{sec3}
\subsection{Newton polyhedron of a polynomial}\label{nondegenerate}
During this article, we use the symbol $[d]$ for the set $\{1,\dots,d\}$, for $d$ in $\mathbb N^*$. Let $x=(x_1,\dots,x_d)$ be a set of $d$ variables, and let $f(x)=\sum_{\alpha\in\mathbb{N}^d}c_{\alpha}x^{\alpha}$ be in $\mathbb C[x]$ with $f(O)=0$, where $O$ is the origin of $\mathbb C^d$. Let $\Gamma$ be the Newton polyhedron of $f$, i.e., the convex hull of the set $\bigcup_{c_{\alpha}\not=0}(\alpha+\mathbb{R}_{\geq 0}^d)$ in $\mathbb{R}_{\geq 0}^d$. Let $F$, resp. $K$, denote the set of all the faces, resp. the compact faces, of $\Gamma$. For every face $\gamma$ of $\Gamma$ (not necessarily compact, the case $\gamma=\Gamma$ included), we define by $f_{\gamma}(x)=\sum_{\alpha\in\gamma}c_{\alpha}x^{\alpha}$ the {\it face function} of $f$ with respect to $\gamma$. The polynomial $f$ is called  {\it nondegenerate on the face  $\gamma\in F$} if  $f_{\gamma}$ has no singular point in $\mathbb{G}_{m,\mathbb C}^d$. We say that $f$ is {\it nondegenerate in the sense of Kouchnirenko} if it is nondegenerate on every compact face $\gamma\in K$. If $f$ is nondegenerate on every face of $\Gamma$ (including non-compact faces, and $\Gamma$ itself), we say that $f$ is {\it nondegenerate in the sense of Newton polyhedron} or simply {\it Newton nondegenerate}. Consider the function 
$\ell=\ell_{\Gamma}: \mathbb R_{\geq 0}^d \to \mathbb R$
which sends $a$ in $\mathbb R_{\geq 0}^d$ to $\inf_{b\in\Gamma}\langle a,b\rangle$,
where $\langle \bullet,\bullet\rangle$ is the standard inner product in $\mathbb R^d$. For $a$ in $\mathbb{R}_{\geq 0}^d$, we denote by $\gamma_a$ the face of $\Gamma$ to which the restriction of the function $\langle a,\bullet\rangle$ gets its minimum, i.e., $b\in\Gamma$ is in $\gamma_a$ if and only if $\langle a,b\rangle=\ell(a)$. Note that $\gamma_a$ is a compact face if and only if $a$ is in $\mathbb{R}_{>0}^d$. Moreover, $\gamma_a=\Gamma$ when $a=(0,\dots,0)$ in $\mathbb R^d$, and $\gamma_a$ is a proper face of $\Gamma$ otherwise. For every proper face $\gamma$ of $\Gamma$, we define 
$$\sigma_{\gamma}:=\sigma_{[d],\gamma}:=\{a\in\mathbb{R}_{\geq 0}^d\mid \gamma=\gamma_a\}.$$ 
It is clear that $\sigma_{\gamma}$ is a cone of dimension $d-\dim(\gamma)$. 

For any $J\subseteq [d]$, put $\mathbb{A}_{\mathbb C}^J:=\Spec \left(\mathbb C\left[(x_i)_{i\in J}\right]\right)$ and $f^J:=f|_{\mathbb A_{\mathbb C}^J}$. If $f$ is nondegenerate in the sense of Kouchnirenko (resp. Newton nondegenerate) then $f^J$ is also nondegenerate in the sense of Kouchnirenko (resp. Newton polyhedron). If $\gamma$ is a proper face of $\Gamma(f^J)$, we denote by $\sigma_{J,\gamma}$ the cone $\{a\in\mathbb{R}_{\geq 0}^J\mid \gamma=\gamma_a\}$, which has the dimension $|J|-\dim(\gamma)$.

\subsection{Contact loci}\label{sect3.2}
Let $f(x_1,\dots,x_d)$ be as above. For $n\in \mathbb N^*$, $k\in \mathbb N$ and $J\subseteq [d]$, we denote by $\Delta_J^{(n,k)}$ the set of $a\in \{0,\dots,n\}^J$ such that $\ell_J(a)+k=n$, where $\ell_J$ stands for $\ell_{\Gamma(f^J)}$. For  $a\in \Delta_J^{(n,k)}$, put
$$\mathscr{X}_{J,a}^{(n)}:=\left\{\varphi\in\mathscr{X}_n(f)\mid \ord_tx_j(\varphi)=a_j \ \forall j\in J,\ x_i(\varphi)\equiv 0\ \forall  i\not\in J\right\}.$$ 
This subvariety of $\mathscr{X}_n(f)$ is invariant by the $\mu_n$-action given by $\xi\cdot\varphi(t)=\varphi(\xi t)$, and it defines an element $[\mathscr{X}_{J,a}^{(n)}]:=[\mathscr{X}_{J,a}^{(n)}\to X_0]$ in $K_0^{\hat\mu}(\Var_{X_0})$, where the structure map is given by $\varphi\mapsto \varphi(0)$. Let $\mathcal{P}_n$ be the index set consisting of all such pairs $(J,a)$ such that 
\begin{align}\label{decompXn}
\mathscr X_{n}(f)=\bigsqcup_{(J,a)\in \mathcal{P}_n}\mathscr{X}_{J,a}^{(n)}.
\end{align} 
Note that for every face $\gamma$ of $\Gamma$ (including $\Gamma$ itself), there exists a unique set $J_{\gamma}\subseteq [d]$ such that $\gamma$ is contained in the hyperplanes $x_j=0$ for all $j\not\in J_{\gamma}$ and not contained in other coordinate hyperplanes. By this, the index set $\mathcal P_n$ in (\ref{decompXn}) is nothing else than the set of pairs $(J_{\gamma},a)$ such that $\gamma\in F$ and $a\in \bigsqcup_{k\in \mathbb N}\left(\sigma_{J_{\gamma},\gamma}\cap \Delta_{J_{\gamma}}^{(n,k)}\right)$.

In particular, if  $f$ is Newton nondegenerate, the $\mathbb C$-subvariety $X_{J}(0)$ of $\mathbb G_{m,\mathbb C}^J$ defined by $f^J(x)$ is smooth for every $J\subseteq [d]$. Thus, we have a $\hat\mu$-equivariant Zariski locally trivial fibration $\mathscr{X}_{J,(0,\dots,0)}^{(n)}\to X_{J}(0)$ with fiber $\mathbb A_{\mathbb C}^{(|J|-1)n}$ (proving this statement is part of the proof of Theorem \ref{mainthm1} below). 

For every face $\gamma\in F$, let us consider the $\mathbb C$-varieties
$$
X_{\gamma}(1):=\big\{x\in \mathbb G_{m,\mathbb C}^{J_{\gamma}} \mid f_{\gamma}(x)=1\big\}\quad \text{and} \quad X_{\gamma}(0):=\big\{x\in \mathbb G_{m,\mathbb C}^{J_{\gamma}} \mid f_{\gamma}(x)=0\big\}.
$$
We always consider the trivial action of $\hat\mu$ on the variety $X_{\gamma}(0)$.  Let $a$ be in the relative interior $\textrm{rel.int.}\sigma_{\gamma}$ of the dual cone $\sigma_{\gamma}$ of $\gamma$, then $\gamma= \gamma_a$.   If $\gamma_a$ is compact, the variety $X_{\gamma}(1)$ admits a natural $\mu_{\ell_{J_{\gamma}}(a)}$-action as follows  
$$e^{2\pi ir/\ell_{J_{\gamma}}(a)}\cdot (x_j)_{j\in J_{\gamma}}:=\big(e^{2\pi ira_j/\ell_{J_{\gamma}}(a)}x_j\big)_{j\in J_{\gamma}},$$
for $r\in [\ell_{J_{\gamma}}(a)]$. Let $s=s_J$ denote the sum function: $s(a)=\sum_{j\in J}a_j$ for $a=(a_j)_{j\in J}\in \mathbb R^J$.

\begin{theorem}\label{mainthm1}
Assume that $f\in \mathbb{C}[x]$ is nondegenerate on a face $\gamma\in F$. If $a\in \sigma_{J_{\gamma},\gamma}\cap \Delta_{J_{\gamma}}^{(n,0)}$ (hence $n=\ell_{J_{\gamma}}(a)$), there is naturally a $\mu_n$-equivariant isomorphism of $\mathbb{C}$-varieties 
$$\tau: \mathscr{X}_{J_{\gamma},a}^{(n)}\to X_{\gamma}(1) \times_{\mathbb C}\mathbb A_{\mathbb C}^{|J_{\gamma}|\ell_{J_{\gamma}}(a)-s(a)}.$$ 
If $k\in \mathbb N^*$ and $a\in \sigma_{J_{\gamma},\gamma}\cap \Delta_{J_{\gamma}}^{(n,k)}$, there is a Zariski locally trivial fibration 
$$\pi: \mathscr{X}_{J_{\gamma},a}^{(n)}\to X_{\gamma}(0)$$
with fiber $\mathbb A_{\mathbb C}^{|J_{\gamma}|(\ell_{J_{\gamma}}(a)+k)-s(a)-k}$. 

As a consequence, the identities $\left[\mathscr{X}_{J_{\gamma},a}^{(n)}\right]=\left[X_{\gamma}(1)\right]\mathbb{L}^{|J_{\gamma}|\ell_{J_{\gamma}}(a)-s(a)}$ for $a\in \sigma_{J_{\gamma},\gamma}\cap \Delta_{J_{\gamma}}^{(n,0)}$, and $ \left[\mathscr{X}_{J_{\gamma},a}^{(n)}\right]=\left[X_{\gamma}(0)\right]\mathbb{L}^{|J_{\gamma}|(\ell_{J_{\gamma}}(a)+k)-s(a)-k}$ for $k\in \mathbb N^*$ and $a\in \sigma_{J_{\gamma},\gamma}\cap \Delta_{J_{\gamma}}^{(n,k)}$ hold in $\mathscr M_{\mathbb{C}}^{\hat\mu}$.
\end{theorem}

\begin{proof}
It suffices to prove the theorem for $J_{\gamma}=[d]$. Let $a=(a_1,\dots,a_d)$ be in $\sigma_{\gamma}\cap \Delta_{[d]}^{(n,0)}$, hence $n=\ell(a)$ and $\gamma=\gamma_a$. Every element $\varphi$ in $\mathscr{X}_{[d],a}^{(n)}$ has the form $\big(\sum_{j=a_1}^{\ell(a)}b_{1j}t^j,\dots, \sum_{j=a_d}^{\ell(a)}b_{dj}t^j\big)$ with $b_{ia_i}\not=0$ for $1\leq i\leq d$. The coefficient of $t^{\ell(a)}$ in $f(\varphi(t))$ is nothing but $f_{\gamma_a}(b_{1a_1},\dots,b_{da_d})$, thus $(b_{1a_1},\dots,b_{da_d})$ is in $X_{\gamma_a}(1)$. We deduce that $\mathscr{X}_{[d],a}^{(\ell(a))}$ is $\mu_{\ell(a)}$-equivariant isomorphic to $X_{\gamma_a}(1)\times_{\mathbb C}\mathbb{A}_{\mathbb C}^{d\ell(a)-s(a)}$ (where the group $\mu_{\ell(a)}$ acts trivially on $\mathbb{A}_{\mathbb C}^{d\ell(a)-s(a)}$) via the map 
$$\theta: \varphi(t)\mapsto \left((b_{ia_i})_{1\leq i\leq d},(b_{ij})_{1\leq i\leq d,a_i<j\leq \ell(a)}\right).$$
Indeed, for every $\xi$ in $\mu_{\ell(a)}$, the element $\varphi(\xi t)$ is sent to $\left((\xi^{a_i}b_{ia_i})_{1\leq i\leq d},(b_{ij})_{1\leq i\leq d,a_i<j\leq \ell(a)}\right)$ which equals $\xi\cdot\left((b_{ia_i})_{1\leq i\leq d},(b_{ij})_{1\leq i\leq d,a_i<j\leq \ell(a)}\right)$. Thus $\theta$ is a $\mu_{\ell(a)}$-equivariant isomorphism. 

Now we prove the second statement. Let $a$ be in $\sigma_{J_{\gamma},\gamma}\cap \Delta_{J_{\gamma}}^{(n,k)}$ for $k\in \mathbb N^*$, hence $n=\ell(a)+k$ and $\gamma=\gamma_a$. For $\varphi$ in $\mathscr X_{[d],a}^{n}$, putting 
\begin{align}\label{tildephi}
\widetilde{\varphi}:=\big(t^{-a_1}x_1(\varphi),\dots,t^{-a_d}x_d(\varphi)\big),
\end{align}
we get
\begin{equation}\label{onestar}
f(\varphi)=t^{\ell(a)}f_{\gamma_a}(\widetilde{\varphi})+\sum_{k\geq 1}t^{\ell(a)+k}\sum_{\langle \alpha,a \rangle=\ell(a)+k}c_{\alpha}\widetilde{\varphi}^{\alpha}.
\end{equation}
Defining
$$\widetilde f(\widetilde{\varphi},t):=f_{\gamma_a}(\widetilde{\varphi})+\sum_{k\geq 1}t^k\sum_{\langle \alpha,a \rangle=\ell(a)+k}c_{\alpha}\widetilde{\varphi}^{\alpha},$$
we obtain a function
$$\widetilde f: \mathscr L_{\ell(a)+k+1-a_1}(\mathbb A_{\mathbb C}^1)\times_{\mathbb C}\cdots \times_{\mathbb C}\mathscr L_{\ell(a)+k+1-a_d}(\mathbb A_{\mathbb C}^1)\times_{\mathbb C} \mathbb A_{\mathbb C}^1 \to \mathbb A_{\mathbb C}^1$$
given by $\widetilde f(\widetilde{\varphi},t_0):=\widetilde f(\widetilde{\varphi}(t_0),t_0)$. It thus follows from (\ref{onestar}) that $\varphi$ is in $\mathscr X_{[d],a}^{\ell(a)+k}$ if and only if $\widetilde f(\widetilde{\varphi},t)=t^k\mod t^{k+1}$. Putting $\widetilde{\varphi}_i=\sum_{j=0}^{\ell(a)-a_i+k}b_{ij}t^j$ for $1\leq i\leq d$, the latter means that 
$$
\begin{cases}
f_{\gamma_a}(b_{10},\dots,b_{d0})=0 & \quad \text{with} \ b_{i0}\not=0 \ \text{for}\ 1\leq i\leq d,\\
q_j(b_{1j},\dots,b_{d_0j})+p_j((b_{i'j'})_{i',j'})=0 & \quad \text{for}\  1\leq j\leq k-1,\\
q_k(b_{1k},\dots,b_{d_0k})+p_k((b_{i'j'})_{i',j'})=1,
\end{cases}
$$
where $p_j$, for $1\leq j\leq k$, are polynomials in variables $b_{i'j'}$ with $i^{'}\leq d_0$ and $j'<j$, and
$$q_j(b_{1j},\dots,b_{d_0j})=\sum_{i=1}^{d_0}\frac{\partial f_{\gamma_a}}{\partial x_i}(b_{10},\dots,b_{d_00}, 0, \ldots, 0)b_{ij}.$$
Note that the function $f$ does not depend on $x_i$ for all $i>d_0$.

We consider the morphism $\pi: \mathscr X_{[d],a}^{(\ell(a)+k)}\to X_{\gamma_a}(0)$ which sends the $\varphi$ described previously to $(b_{10},\dots,b_{d_0})$. Since $\hat\mu$ acts trivially on $X_{\gamma_a}(0)$, we only need to prove that $\pi$ is a locally trivial fibration with fiber $\mathbb A_{\mathbb C}^{d(\ell(a)+k)-s(a)-k}$. For every $1\leq i\leq d_0$, we put
\begin{align}\label{trivialization}
U_i:=\Big\{(x_1,\dots,x_d)\in X_{\gamma_a}(0) \mid \frac{\partial f_{\gamma_a}}{\partial x_i}(x_1,\dots,x_d)\not=0\Big\}.
\end{align} 
The nondegeneracy of $f$ on the face $\gamma=\gamma_a$ gives us an open covering $\{U_1,\dots,U_{d_0}\}$ of $X_{\gamma}(0)$. We construct trivializations of $\pi$ as follows
$$\xymatrix{
\pi^{-1}(U_i) \ar[rr]^{\Phi_{U_i}}\ar@{->}[dr]_{\pi}&& U_i\times_{\mathbb{C}} \mathbb{A}_{\mathbb{C}}^{e}\ar@{->}[dl]^{\pr_1}\\
&U_i&       }
$$
where $e=\sum_{i=1}^d(\ell(a)-a_i+k)-k$ and we identify $\mathbb{A}_{\mathbb{C}}^{e}$ with the subvariety of $\mathbb{A}_{\mathbb{C}}^{\sum_{i=1}^d(\ell(a)-a_i+k)}$ defined by the equations $\widetilde b_{ij}=0$ for $1\leq j\leq k-1$ and $\widetilde b_{ik}=1$ in the coordinate system $(\widetilde b_{lj})$, and for $\varphi$ as previous, 
$$\Phi_{U_i}(\varphi)=((\widetilde b_{10},\dots,\widetilde b_{d0}),(\widetilde b_{lj})_{1\leq l\leq d, 1\leq j\leq \ell(a)-a_l+k}),$$
with $\widetilde b_{ij}=0$ if $1\leq j\leq k-1$, $\widetilde b_{ik}=1$, and $\widetilde b_{lj}=b_{lj}$ otherwise. Furthermore, the inverse map $\Phi_{U_i}^{-1}$ of $\Phi_{U_i}$ is also a regular morphism given explicitly as follows
$$\Phi_{U_i}^{-1}(\widetilde b_{lj})=\Big(\sum_{j=0}^{\ell(a)-a_1+k}b_{1j}t^{j+a_1},\dots,\sum_{j=0}^{\ell(a)-a_d+k}b_{dj}t^{j+a_d}\Big),$$
where $b_{lj}=\widetilde b_{lj}$ for either that $l\not=i$ or that $l=i$ and $k<j\leq \ell(a)-a_l+k$, and
$$b_{ij}=\frac{-p_j((b_{lj'})_{l\leq d_0, j'<j})-\sum_{l\leq d_0, l\not=i}(\partial f_{\gamma_a}/\partial x_l)(\widetilde b_{10},\dots,\widetilde b_{d0})\widetilde b_{lj}}{(\partial f_{\gamma_a}/\partial x_i)(\widetilde b_{10},\dots,\widetilde b_{d0})},$$
for $1\leq j\leq k-1$, and
$$b_{ik}=\frac{1-p_k((b_{lj'})_{l\leq d_0, j'<k})-\sum_{l\leq d_0, l\not=i}(\partial f_{\gamma_a}/\partial x_l)(\widetilde b_{10},\dots,\widetilde b_{d0})\widetilde b_{lk}}{(\partial f_{\gamma_a}/\partial x_i)(\widetilde b_{10},\dots,\widetilde b_{d0})}.$$
This proves that $\pi$ is a locally trivial fibration with fiber $\mathbb A_{\mathbb C}^e$.
\end{proof}


\subsection{Motivic nearby cycles}
We still use the notation introduced previously. In particular, we denotes by $F$ (resp. $K$) the set of all the faces (resp. compact faces) of $\Gamma$. For $0\leq m\leq d$, we denote by $F(m)$ the set of $\gamma\in F$ such that if $\gamma=\gamma_a$ then $a_i\geq 1$ for all $i\in J_{\gamma}\cap [m+1,d]$. Note that $F(0)=K$.

\begin{theorem}\label{thm4.1}
Let $f\in \mathbb C[x_1,\dots,x_d]$, and let $d_1, d_2 \in\mathbb N$ such that $d=d_1+d_2$. The below identities hold in $\mathscr{M}_{\mathbb C}^{\hat\mu}$.
\begin{itemize}
\item[(i)] If $f$ is Newton nondegenerate and if there is a closed immersion $\iota: \mathbb A_{\mathbb C}^{d_1}\hookrightarrow X_0$, then
$$\int_{\mathbb A_{\mathbb C}^{d_1}}\iota^*\mathscr{S}_f=\sum_{\gamma\in F(d_1)}(-1)^{|J_{\gamma}|+1-\dim(\gamma)}\left([X_{\gamma}(1)]-[X_{\gamma}(0)]\right).$$ 
		
\item[(ii)] If $f$ is nondegenerate in the sense of Kouchnirenko and $f(O)=0$, then
$$\mathscr{S}_{f,O}=\sum_{\gamma\in K}(-1)^{|J_{\gamma}|+1-\dim(\gamma)}\left([X_{\gamma}(1)]-[X_{\gamma}(0)]\right).$$ 
\end{itemize}
\end{theorem}

\begin{proof}
(i) By the decomposition (\ref{decompXn}) and by Theorem \ref{mainthm1}, we have 
\begin{align*}
&\int_{\mathbb A_{\mathbb C}^{d_1}}\iota^*[\mathscr X_n(f)]\L^{-nd}\\
&\quad =\sum_{\gamma\in F(d_1)}\Big(\sum_{(=)}[X_{\gamma}(1)]\mathbb{L}^{(|J_{\gamma}|-d)\ell_{J_{\gamma}}(a)-s(a)}+\sum_{(<)}[X_{\gamma}(0)]\mathbb{L}^{(|J_{\gamma}|-d)(\ell_{J_{\gamma}}(a)+k)-s(a)-k}\Big),
\end{align*}
where the sum $\sum_{(=)}$ runs over the pairs $(J_{\gamma},a)$ with $a\in \sigma_{J_{\gamma},\gamma}\cap \Delta_{J_{\gamma}}^{(n,0)}$, and the sum $\sum_{(<)}$ runs over the pairs $(J_{\gamma},a)$ with $a\in \sigma_{J_{\gamma},\gamma}\cap \Delta_{J_{\gamma}}^{(n,k)}$ for $k\geq 1$. By the rationality of $Z_f(T)$ proved in \cite{DL1}, we see that $\int_{\mathbb A_{\mathbb C}^{d_1}}\iota^*$ commutes with the sum $\sum_{n\geq 1}$ in $Z_f(T)$, thus
$$\int_{\mathbb A_{\mathbb C}^{d_1}}\iota^*Z_f(T)=\sum_{\gamma\in F(d_1)}\Big([X_{\gamma}(1)]+[X_{\gamma}(0)]\frac{\mathbb L^{|J_{\gamma}|-d-1}T}{1-\mathbb L^{|J_{\gamma}|-d-1}T}\Big)    \sum_{a\in\sigma_{J_{\gamma},\gamma}}\mathbb{L}^{-s(a)}(\mathbb{L}^{|J_{\gamma}|-d}T)^{\ell_{J_{\gamma}}(a)}.$$
By \cite[Lemme 2.1.5.]{G} we have 
$$\lim_{T\to \infty}\sum_{a\in\sigma_{J_{\gamma},\gamma}}\mathbb{L}^{-s(a)}(\mathbb{L}^{|J_{\gamma}|-d}T)^{\ell_{J_{\gamma}}(a)}=(-1)^{\dim(\sigma_{J_{\gamma},\gamma})}=(-1)^{|J_{\gamma}|-\dim(\gamma)},$$
from which 
$$\int_{\mathbb A_{\mathbb C}^{d_1}}\iota^*\mathscr{S}_f=-\lim\limits_{T\to \infty}\int_{\mathbb A_{\mathbb C}^{d_1}}\iota^*Z_f(T)=\sum_{\gamma\in F(d_1)}(-1)^{|J_{\gamma}|+1-\dim(\gamma)}\left([X_{\gamma}(1)]-[X_{\gamma}(0)]\right).$$
	
(ii) The study is local at $O$, so we only work with all $a\in \sigma_{J_{\gamma},\gamma}\cap \Delta_{J_{\gamma}}^{(n,k)}$ for $\gamma\in K$, i.e. we only need the condition that $f$ is nondegenerate in the sense of Kouchnirenko. Then the argument in (i) runs for $d_1=0$, thus (ii) follows. 
\end{proof}

\begin{remark}
This result revisits Guibert's work in \cite[Section 2.1]{G} for Newton nondegenerate polynomials $f$ in a more general setting. Indeed, in \cite{G} Guibert requires $f$ to have the form $\sum_{\nu\in \mathbb N_{>0}^d}a_{\nu}x^{\nu}$, while we do not. Recently, Bultot-Nicaise in \cite[Theorems 7.3.2, 7.3.5]{BN} provide a new approach to the motivic zeta functions $Z_f(T)$ and $Z_{f,O}(T)$ using log smooth models. 
\end{remark}


We now consider the relation between the motivic nearby cycles of $f$ and that of a restriction of $f$. We write $\widetilde f$ for $f^{[d-1]}$, i.e., $\widetilde f(x_1,\ldots,x_{d-1})=f(x_1,\ldots,x_{d-1},0)$, and write $\widetilde O$ for the origin of $\mathbb C^{d-1}$. Let $\widetilde X_0$ be the zero locus of $\widetilde f$, which can be embedded into $X_0$. The following theorem may be partially considered as a motivic analogue of D. T. L\^e's work on a monodromy relation of a complex singularity and its restriction to a generic hyperplane (see \cite{LDT}). 

\begin{theorem}\label{restricting}
Let $f\in \mathbb C[x_1,\dots,x_d]$, and let $d_1, d_2 \in\mathbb N$ such that $d=d_1+d_2$. The below identities hold in $\mathscr{M}_{\mathbb C}^{\hat\mu}$.
\begin{itemize}
\item[(i)] Suppose that $f$ is Newton nondegenerate and that $\mathbb A_{\mathbb C}^{d_1}$ is embedded in $\widetilde X_0\subseteq X_0$. Denote by $\iota$ the inclusion of $\mathbb A_{\mathbb C}^{d_1}$ in both $X_0$ and $\widetilde X_0$. Then 
$$\int_{\mathbb A_{\mathbb C}^{d_1}}\iota^*\mathscr{S}_f=\int_{\mathbb A_{\mathbb C}^{d_1}}\iota^*\mathscr{S}_{\widetilde{f}}+\int_{\mathbb A_{\mathbb C}^{d_1}}\iota^*\mathscr S_{f,x_d}^{\Delta}.$$ 
		
\item[(ii)] If $f$ is nondegenerate in the sense of Kouchnirenko and $f(O)=0$, then
$$\mathscr{S}_{f,O}=\mathscr{S}_{\widetilde f,\widetilde O}+\mathscr S_{f,x_d,O}^{\Delta}.$$ 
\end{itemize}
\end{theorem}


\begin{proof}
As in the proof of Theorem \ref{thm4.1}, the proof method for (ii) is the same as for (i) but with $d_1=0$. So, we are only going to prove (i). By the definition of $(n,m)$-iterated contact loci, we have
\begin{align*}
\mathscr{X}_{n,m,O}(f,x_d)=\bigsqcup_{(J_{\gamma},a)\in \mathcal P_n. a_d=m}\mathscr X_{J_{\gamma},a}^{(n)},
\end{align*}
We deduce from Section \ref{remark} and the method in the proof of Theorem \ref{thm4.1} that $\int_{\mathbb A_{\mathbb C}^{d_1}}\iota^*Z_{f,x_d}^{\Delta}(T)$ is equal in $\mathscr M_{\mathbb C}^{\hat\mu}[[T]]$ to 
$$\sum_{\gamma\in F(d_1)}\sum_{\begin{smallmatrix} a\in \sigma_{J_{\gamma},\gamma}\\ \ell_{J_{\gamma}}(a)\geq a_d\geq 1 \end{smallmatrix}}\left[\mathscr X_{J_{\gamma},a}^{(\ell_{J_{\gamma}}(a))}\right]\Lbb^{-d\ell_{J_{\gamma}}(a)}T^{\ell_{J_{\gamma}}(a)}$$
plus
$$\sum_{\gamma\in F(d_1)}\sum_{\begin{smallmatrix} a\in \sigma_{J_{\gamma},\gamma}\\ \ell_{J_{\gamma}}(a)+k\geq a_d\geq 1, k\geq 1 \end{smallmatrix}}\left[\mathscr X_{J_{\gamma},a}^{(\ell_{J_{\gamma}}(a)+k)}\right]\Lbb^{-d(\ell_{J_{\gamma}}(a)+k)}T^{\ell_{J_{\gamma}}(a)+k}.$$
We apply Theorem \ref{mainthm1} to $\gamma\in F(d_1)$ and $a\in \sigma_{J_{\gamma},\gamma}$. If $d\in J_{\gamma}$, then $\ell_{J_{\gamma}}(a)+k\geq a_d\geq 1$ automatically for any $k\in \mathbb N$. If $d\not\in J_{\gamma}$, then the inequalities $\ell_{J_{\gamma}}(a)+k\geq a_d\geq 1$ is in the situation of \cite[Lemma 2.10]{GLM1}, in which the corresponding series has the limit zero. Thus
$$\int_{\mathbb A_{\mathbb C}^{d_1}}\iota^*\mathscr S_{f,x_d}^{\Delta}=\sum_{\gamma\in F(d_1), d\in J_{\gamma}}(-1)^{|J_{\gamma}|+1-\dim(\gamma)}\left([X_{\gamma}(1)]-[X_{\gamma}(0)]\right),$$
from which, by Theorem \ref{thm4.1},  
$$\int_{\mathbb A_{\mathbb C}^{d_1}}\iota^*\mathscr S_{f}=\sum_{\gamma\in F(d_1), d\not\in J_{\gamma}}(-1)^{|J_{\gamma}|+1-\dim(\gamma)}\left([X_{\gamma}(1)]-[X_{\gamma}(0)]\right)+\int_{\mathbb A_{\mathbb C}^{d_1}}\iota^*\mathscr S_{f,x_d}^{\Delta}.$$
The condition $d\not\in J_{\gamma}$ means that $J_{\gamma}\subseteq [d-1]$, hence the first sum $\sum_{\gamma\in K, d\not\in J_{\gamma}}$ in the above decomposition of $\mathscr S_{f,O}$ is nothing but $\mathscr{S}_{\widetilde f,\widetilde O}$, again by Theorem \ref{thm4.1}.
\end{proof}

To state and prove the below corollary, we denote by $O_l$ the origin of the affine space $\mathbb C^{d_l}$ for $1\leq l \leq 3$.

\begin{corollary}[Integral identity conjecture]\label{IIC}
Let $(x,y,z)$ be the standard coordinates of the affine space $\mathbb C^d=\mathbb C^{d_1}\times \mathbb C^{d_2}\times \mathbb C^{d_3}$. Let $f$ be a Newton nondegenerate polynomial in $\mathbb C[x,y,z]$ such that $f(O)=0$ and $f(\lambda x, \lambda^{-1} y,z)=f(x,y,z)$ for all $\lambda$ in $\mathbb C^*$. Then the integral identity $\int_{\mathbb A_{\mathbb C}^{d_1}}\iota^*\mathscr{S}_f=\L^{d_1}\mathscr{S}_{f|_{\mathbb C^{d_3}},O_3}$ holds in $\mathscr M_{\mathbb C}^{\hat\mu}$, with $\iota$ the inclusion of $\mathbb A_{\mathbb C}^{d_1}$ in $X_0$.
\end{corollary}

\begin{proof}
Put $f_0:=f$; and for $1\leq j\leq d_2$, let $f_j$ be the restriction of $f_{j-1}$ to $y_j^{-1}(0)$. For abuse of notation we use $\iota$ commonly for the inclusions of $\mathbb A_{\mathbb C}^{d_1}$ in the varieties $f_j^{-1}(0)$. By applying Theorem \ref{restricting} (several times for part (i)) we have
$$\int_{\mathbb A_{\mathbb C}^{d_1}}\iota^*\mathscr S_f=\int_{\mathbb A_{\mathbb C}^{d_1}}\iota^*\mathscr S_{f_{d_2}}+\sum_{j=1}^{d_2}\int_{\mathbb A_{\mathbb C}^{d_1}}\iota^*\mathscr S_{f_{j-1},y_j}^{\Delta}.$$
The hypothesis on $f$ implies that $f(x,0,z)=f(0,0,z)$, hence
$$\int_{\mathbb A_{\mathbb C}^{d_1}}\iota^*\mathscr S_{f_{d_2}}=\L^{d_1}\mathscr{S}_{f|_{\mathbb C^{d_3}},O_3}.$$ 
It thus remains to prove that $\int_{\mathbb A_{\mathbb C}^{d_1}}\iota^*\mathscr S_{f_{j-1},y_j}^{\Delta}=0$ in $\mathscr M_{\mathbb C}^{\hat\mu}$, for every $1\leq j\leq d_2$, and it suffices to check this with $j=1$. Using the proof of Theorem \ref{restricting} we have
\begin{align}\label{IICf}
\int_{\mathbb A_{\mathbb C}^{d_1}}\iota^*\mathscr{S}_{f,y_1}^{\Delta}=\sum_{\gamma\in F(d_1), d_1+1\in J_{\gamma}}(-1)^{|J_{\gamma}|+1-\dim(\gamma)}\left([X_{\gamma}(1)]-[X_{\gamma}(0)]\right).
\end{align}
For all $\gamma\in F(d_1)$ with $J_{\gamma}$ containing $d_1+1$, write $J_{\gamma}=J_1\sqcup J_2\sqcup J_3$ with $J_1\subseteq [d_1]$, $J_2\subseteq d_1+[d_2]$ and  $J_3\subseteq d_1+d_2+[d_3]$. Observe that $J_2$ is nonempty because $d_1+1\in J_2$, it thus implies from the hypothesis on $f$ that $J_1$ is also nonempty. Write $\gamma=\widetilde\gamma +\mathbb R_{\geq 0}^I$ with $\widetilde\gamma$ a compact face, then $I\subseteq [d_1]$, $J_{\gamma}=J_{\widetilde\gamma}$ and $\dim(\gamma)=|I|+\dim(\widetilde\gamma)$. Since $f(\lambda x,\lambda^{-1}y,z)=f(x,y,z)$ for all $\lambda\in \mathbb C^*$, we can write $f$ in the form
$$f(x,y,z)=\sum_{s(\alpha)=s(\beta)}a_{\alpha\beta\gamma}x^{\alpha}y^{\beta}z^{\gamma}.$$
Since the hyperplane $\{s(\alpha)=s(\beta)\}$ in $\mathbb R_{>0}^d$ goes through the origin, its intersection with $\partial\Gamma$ has codimension at least $2$. This together with $J_1\not=\emptyset$ implies that $I\not=\emptyset$. Using Lemmas 3.2, 3.3, 4.7 in \cite{Thuong1} we can rewrite (\ref{IICf}) as follows
\begin{align*}
\int_{\mathbb A_{\mathbb C}^{d_1}}\iota^*\mathscr{S}_{f,y_1}^{\Delta}=\sum_{\widetilde\gamma\in K}(-1)^{|J_{\widetilde\gamma}|+1-\dim(\widetilde\gamma)}\left([X_{\widetilde\gamma}(1)]-[X_{\widetilde\gamma}(0)]\right)\sum_{I\subseteq M_{\widetilde\gamma}}(-1)^{|I|},
\end{align*}
in which there exists for each $\widetilde\gamma \in K$ a unique $M_{\widetilde\gamma}\not=\emptyset$ such that the previous identity holds. Since $\sum_{I\subseteq M_{\widetilde\gamma}}(-1)^{|I|}=0$, we get $\int_{\mathbb A_{\mathbb C}^{d_1}}\iota^*\mathscr{S}_{f,y_1}^{\Delta}=0$ in $\mathscr M_{\mathbb C}^{\hat\mu}$. The corollary has been proved.
\end{proof}

\begin{remark}
In fact, the integral identity conjecture states for formal series. It plays a crucial role in Kontsevich-Soibelman's theory of motivic Donaldson-Thomas invariants for noncommutative Calabi-Yau threefolds (see \cite{KS}). We refer to \cite{Thuong1}, \cite{Thuong}, \cite{NP}, \cite{LN} and \cite{Thuong2} for proofs of different versions of the conjecture.
\end{remark}


\section{Cohomology groups of contact loci of nondegenerate singularities}\label{Sec5-cloci}
As before, let $f$ be in $\mathbb C[x_1,\dots,x_d]$ which vanishes at $O$. In this section, we always assume that $f$ is nondegenerate in the sense of Kouchnirenko (say for short that $f$ is {\it nondegenerate}).

\subsection{Borel-Moore homology groups of contact loci}
Consider the following decomposition of $\mathscr X_{n,O}(f)$, which is the local version of (\ref{decompXn}),
\begin{align}\label{otherdecomp}
 \mathscr X_{n,O}(f)=\bigsqcup_{(J,a)\in \widetilde{\mathcal P}_n}\mathscr{X}_{J,a}^{(n)},
\end{align}
where $\widetilde{\mathcal P}_n$ is the set of all the pairs $(J_{\gamma},a)$ such that $\gamma \in K$ and $a\in \bigsqcup_{k\in \mathbb N}\left(\sigma_{J_{\gamma},\gamma}\cap \Delta_{J_{\gamma}}^{(n,k)}\right)$. Remark that, as $\gamma \in K$, we have $\sigma_{J_{\gamma},\gamma}\cap \Delta_{J_{\gamma}}^{(n,k)}=(\mathbb N^*)^{J_{\gamma}}\cap\sigma_{J_{\gamma},\gamma}\cap \Delta_{J_{\gamma}}^{(n,k)}$.
We consider an ordering in $\widetilde{\mathcal P}_n$ as follows: for $(J, a)$ and $(J', a')$ in $\widetilde{\mathcal P}_n$, $(J', a')\leq (J, a)$ if and only if $J'\subseteq J$ and $a_i\leq a'_i$ for all $i\in J'$, where $a=(a_i)_{i\in J}$ and $a'=(a'_i)_{i\in J'}$. The following lemma is straightforward.

\begin{lemma}\label{lem5.1}
Let $n$ be in $\mathbb N^*$. For all $(J, a)$ and $(J', a')$ in $\widetilde{\mathcal P}_n$, the following are equivalent:
\begin{align*}
\mathrm{(i)} \ (J', a')\leq (J, a), \qquad\quad \mathrm{(ii)} \ \mathscr{X}_{J',a'}^{(n)}\subseteq \overline{\mathscr X_{J,a}^{(n)}}, \qquad\quad \mathrm{(iii)} \ \mathscr{X}_{J',a'}^{(n)}\cap \overline{\mathscr X_{J,a}^{(n)}}\neq \emptyset,
\end{align*}
the closure taken in the usual topology. Consequently, $\overline{\mathscr X_{J,a}^{(n)}}= \bigsqcup\limits_{(J', a')\leq (J, a)} \mathscr{X}_{J',a'}^{(n)}$ for $(J,a)\in \widetilde{\mathcal P}_n$.
 \end{lemma}
 

 
 
Consider the function $\eta: \widetilde{\mathcal P}_n\to \mathbb{Z}$ given by $\eta(J,a)=\dim_{\mathbb{C}}\mathscr{X}_{J,a}^{(n)}$, for every $n\in \mathbb N^*$. Put 
\begin{align*}
S_p:=  \bigsqcup_{(J, a)\in \widetilde{\mathcal P}_n, \eta(J,a)\leq p}\mathscr{X}_{J,a}^{(n)},
\end{align*}
for $p\in \mathbb N$. The below is a property of $\eta$ and $S_p$'s.

\begin{lemma}\label{lem5.3}
Let $n$ be in $\mathbb N^*$.
\begin{itemize}
\item[(i)]
If $(J', a')\leq (J, a)$ in $\widetilde{\mathcal P}_n$, then $\eta(J', a')\leq \eta(J, a)$.

\item[(ii)] For all $p\in \mathbb N$, $S_p$ are closed and $S_p\subseteq S_{p+1}$. As a consequence, there is a filtration of $\mathscr X_{n,O}(f)$ by closed subspaces:
$$\mathscr X_{n,O}(f)=S_{d_0}\supseteq S_{d_0-1}\supseteq \cdots \supseteq S_{-1}=\emptyset,$$
where $d_0$ denotes the $\mathbb C$-dimension of $\mathscr X_{n,O}(f)$.
\end{itemize}
\end{lemma}
 
\begin{proof}
The first statement (i) is trivial. To prove (ii) we take the closure of $S_p$; then using Lemma \ref{lem5.1} we get 
$$\overline{S}_p=\bigcup_{\eta(J, a)\leq p} \overline{\mathscr X_{J,a}^{(n)}}= \bigcup_{\eta(J, a)\leq p} \bigsqcup_{(J',a')\leq (J, a)}\mathscr{X}_{J',a'}^{(n)}.$$
This decomposition combined with (i) implies that $\overline{S}_p\subseteq S_p$, which proves  that $S_p$ is a closed subspace. The remaining statements of (ii) are trivial. 
\end{proof}
 
A main result of this section is the following theorem. To express the result, we work with the Borel-Moore homology $H_*^{\BM}$. 

\begin{theorem}\label{thm5.4}
Let $f\in \mathbb C[x_1,\dots,x_d]$ be nondegenerate,  $n\in \mathbb N^*$ and $f(O)=0$. Then there is a spectral sequence
$$E_{p,q}^1:=\bigoplus_{(J, a)\in \widetilde{\mathcal P}_n,\eta(J, a)=p}H^{\BM}_{p+q}(\mathscr{X}_{J,a}^{(n)})\Longrightarrow  H^{\BM}_{p+q}(\mathscr X_{n,O}(f)).$$
\end{theorem}

\begin{proof}
We have the following the Gysin exact sequence
$$\cdots \to H^{\BM}_{p+q}(S_{p-1})\to H^{\BM}_{p+q}(S_{p})\to H^{\BM}_{p+q}(S_{p}\setminus S_{p-1})\to H^{\BM}_{p+q-1}(S_{p-1})\to \cdots .$$
Put 
$$A_{p, q}:= H^{\BM}_{p+q}(S_{p}),\quad E_{p,q}:=H^{\BM}_{p+q}(S_{p}\setminus S_{p-1}).$$
Then we have the bigraded $\mathbb{Z}$-modules $A:= \bigoplus_{p,q} A_{p,q}$ and $E:= \bigoplus_{p,q} E_{p,q}$. The previous exact sequence induces the exact couple
$\langle A, E; h, i, j\rangle,$
where $h: A\to A$ is induced from the inclusions $S_m\subseteq S_{m+1}$, $i: A\to E$ and $j: E\to A$ are induced from the above exact sequence. Since the filtration in Lemma \ref{lem5.3} (ii) is finite, that exact couple gives us the following spectral sequence
$$E^1_{p,q}:=E_{p,q}=H^{\BM}_{p+q}(S_{p}\setminus S_{p-1})\Longrightarrow H^{\BM}_{p+q}(\mathscr X_{n,O}(f)).$$

On the other hand, we have
$$S_p\setminus S_{p-1}= \bigsqcup_{\eta(J, a)=p} \mathscr{X}_{J,a}^{(n)}.$$
One claims that for two different pairs $(J, a), (J^{'}, a^{'})$ in $\widetilde{\mathcal P}_n$ which $\eta(J, a)=\eta(J^{'}, a^{'})=p$ then 
$$\mathscr{X}_{J,a}^{(n)}\cap \overline{\mathscr X_{J^{'}, a^{'}}^{(n)}}= \emptyset\quad \textrm{and}\quad 
\mathscr X_{J^{'}, a^{'}}^{(n)}\cap \overline{\mathscr{X}_{J,a}^{(n)}}= \emptyset.$$
Indeed, if otherwise,  suppose that 
$$\mathscr{X}_{J^{'}, a^{'}}^{(n)}\cap \overline{\mathscr X_{J,a}^{(n)}}\neq \emptyset.$$
By Lemma \ref{lem5.1}, we obtain that $\mathscr{X}_{J^{'}, a^{'}}^{(n)}\subseteq \overline{\mathscr X_{J,a}^{(n)}}$, but $\mathscr{X}_{J^{'}, a^{'}}^{(n)}$ and $\mathscr X_{J,a}^{(n)}$ are two disjoint smooth manifolds, then $\eta(J^{'}, a^{'}) < \eta(J, a)$. This is a contradiction. 

Therefore, in the set $S_p\setminus S_{p-1}$ with the induced topology, each set $\mathscr{X}_{J,a}^{(n)}$ which $ \eta(J, a)=p$ is open, hence, is also closed. This implies that 
$$H^{\BM}_{p+q}(S_{p}\setminus S_{p-1})= \bigoplus_{(J, a)\in \widetilde{\mathcal P}_n,\eta(J, a)=p}H^{\BM}_{p+q}(\mathscr{X}_{J,a}^{(n)}).$$
The theorem is then proved.
\end{proof}
 

 \begin{corollary}
With the hypothesis as in Theorem \ref{thm5.4}, there is an isomorphism of groups
$$H^{\BM}_{2d_0}(\mathscr X_{n,O}(f))\cong \mathbb{Z}^s,$$
where $s$ is the number of connected components of $\mathscr X_{n,O}(f)$ which have the same complex dimension $d_0$ as $\mathscr X_{n,O}(f)$.
 \end{corollary}

\subsection{Sheaf cohomology groups of contact loci}
In this subsection, we are going to prove the following theorem.

\begin{theorem}\label{thm5.6}
Let $f\in\mathbb C[x_1,\dots,x_d]$ be nondegenerate, $n\in \mathbb N^*$ and $f(O)=0$. Let $\mathcal F$ be an arbitrary sheaf of abelian groups on $\mathscr X_{n,O}(f)$. Then, there is a spectral sequence
\begin{equation}\label{spseq5.6}
E^{p,q}_1:=\bigoplus_{(J, a)\in \widetilde{\mathcal P}_n,\eta(J, a)=p}H^{p+q}_c(\mathscr{X}_{J,a}^{(n)}, \mathcal{F})\Longrightarrow  H^{p+q}_c(\mathscr X_{n,O}(f), \mathcal{F}).
\end{equation}
\end{theorem}

\begin{proof} 
We use the notation in Lemma \ref{lem5.3}. For simplicity, we write $S$ for $S_{d_0}=\mathscr X_{n,O}(f)$. For any $0\leq p \leq d_0$, we put $S_p^{\circ}:= S_p\setminus S_{p-1},$
which is a $\mu_n$-invariant subset of $S$. Consider the inclusions $j_p: S^{\circ}_p \hookrightarrow S_p$, $k_p: S\setminus S_p \hookrightarrow  S$ and $i_p: S_p\hookrightarrow  S$. Put $\mathcal{F}_p:= (j_{p})_!(j_p)^{-1}(i_p)^{-1}\mathcal{F}$ and $F^p(\mathcal{F}):= (k_{p-1})_!(k_{p-1})^{-1}\mathcal{F}$ for every $p\geq 1$, with the convention $F^0(\mathcal{F}):= \mathcal{F}$. Then we have the exact sequences
$$0\to F^{p+1}(\mathcal{F})\to F^p(\mathcal{F}) \quad\textrm{and}\quad 0\to (i_{p})_*\mathcal{F}_p\to \mathcal{F}|_{S_p}\to \mathcal{F}|_{S_{p-1}},$$
in which by $\mathcal{F}|_{S_p}$ we mean $(i_p)_*(i_p)^{-1}\mathcal F$. 
Therefore we have the following diagram
$$
\left.
\begin{array}{ccccccccc}
  &  &  &  &  &  & 0 &  &  \\
  &  &  &  &  &  & \downarrow &  &  \\
  &  & 0 &  &  &  & (i_{p})_*\mathcal{F}_p &  &  \\
  &  & \downarrow &  &  &  & \downarrow &  &  \\
 0 & \to & F^{p+1}(\mathcal{F}) & \to & \mathcal{F} & \to & \mathcal{F}|_{S_p} & \to  & 0 \\
  &  & \downarrow &  & || &  & \downarrow &  &  \\
 0 & \to & F^{p}(\mathcal{F}) & \to & \mathcal{F} & \to & \mathcal{F}|_{S_{p-1}} & \to &  0
  \end{array}
\right.
$$
It implies from the snake lemma that $F^{p}(\mathcal{F})/F^{p+1}(\mathcal{F})\cong (i_{p})_*\mathcal{F}_p.$ Thus there is a filtration of  $\mathcal{F}$ by ``skeleta'': $\mathcal{F}= F^{0}(\mathcal{F})\supseteq F^{1}(\mathcal{F})\supseteq \cdots.$
It gives the following spectral sequence of cohomology groups with compact support
\begin{equation}\label{eq:specseqYdt}
E_1^{p,q}(S,\mathcal{F}):=  H^{p+q}_c(S,(i_{p})_*\mathcal{F}_p)
\Longrightarrow H^{p+q}_c(S,\mathcal{F}). 
\end{equation}

Since $S_p$ is a closed subset of $S$, $H^{m}_c(S,(i_{p})_*\mathcal{F}_p)\cong H^{m}_c(S_p,\mathcal{F}_p)$ for any $m$ in $\mathbb N$. Also, by the isomorphisms given by the extension by zero sheaf, we have 
$$H^{m}_c(S_p,\mathcal{F}_p)= H^{m}_c(S_p,(j_{p})_!(j_p)^{-1}(i_p)^{-1}\mathcal{F})    \cong H^m_c(S_p^{\circ}, (j_p)^{-1}(i_p)^{-1}\mathcal{F}).$$
We have that 
$$S_p^{\circ}=\bigsqcup\limits_{\eta(J,a)=p} \mathscr{X}_{J,a}^{(n)}.$$
 Then by the reason as in the proof of Theorem \ref{thm5.4}, we get 
$$H^m_c(S_p^{\circ}, (j_p)^{-1}(i_p)^{-1}\mathcal{F})=\bigoplus_{(J, a)\in \widetilde{\mathcal P}_n,\eta(J, a)=p}H^{p+q}_c(\mathscr{X}_{J,a}^{(n)},(l_{J,a})^{-1}(j_p)^{-1}(i_p)^{-1}\mathcal{F}),$$
where $l_{J,a}$ is the inclusion of $\mathscr{X}_{J,a}^{(n)}$ in $S_p^{\circ}$. For simplicity of notation, we write $H^{p+q}_c(\mathscr{X}_{J,a}^{(n)}, \mathcal{F})$ instead of $H^{p+q}_c(\mathscr{X}_{J,a}^{(n)},(l_{J,a})^{-1}(j_p)^{-1}(i_p)^{-1}\mathcal{F})$. The proof is completed.
\end{proof}

 Now, we consider the spectral sequence (\ref{spseq5.6}) for constant sheaf. we need some notation, for each $\gamma\in K$, $k\in \mathbb N, n\in \mathbb N^*$ and $p\in \mathbb Z$, we denote by $D_{\gamma,k,p}^{(n)}$ the set of all $a\in \sigma_{J_{\gamma}, \gamma}\cap \Delta^{(n,k)}_{J_{\gamma}}$ such that $d-1+|J_{\gamma}|n-s(a)-k=p$, which is a finite set. Then, for each $(J, a)\in \widetilde{\mathcal P}_n$ which $\eta(J, a)=p$, there exist $\gamma\in K$ and $k\in \mathbb N$ such that $J= J_{\gamma}$ and $a\in D_{\gamma,k,p}^{(n)}$.  

  The summands in the spectral sequence (\ref{spseq5.6}) are described more explicitly in case of constant sheaf as below.

\begin{lemma} 
Let $\gamma\in K, n\in \mathbb N^*$ and $p, q\in \mathbb Z$. Then, for $J= J_{\gamma}$ and $a\in D_{\gamma,0,p}^{(n)}$ we have
\begin{equation}
H^{p+q}_c(\mathscr{X}_{J,a}^{(n)}, \mathbb{C})  \cong H_{p-q}(X_{\gamma}(1), \mathbb{C}).
\end{equation}

\end{lemma}

\begin{proof}
It follows from Theorem \ref{mainthm1} that $\mathscr{X}_{J,a}^{(n)}$ is a $p$-dimensional complex manifold and is homeomorphic to  $X_{\gamma}(1) \times{\mathbb C}^{|J_{\gamma}|\ell_{J_{\gamma}}(a)-s(a)}$. Then, by combining the duality and the Kunneth formula we get the conclusion.
\end{proof}

We also have the following description for the cohomology of  $\mathscr{X}_{J,a}^{(n)}$ for $J= J_{\gamma}$ and $a\in D_{\gamma,k,p}^{(n)}, k\in \mathbb N^*.$

\begin{lemma} 
Let $\gamma\in K, n, k\in \mathbb N^*$ and $p, q\in \mathbb Z$. Then, for $J= J_{\gamma}$ and $a\in D_{\gamma,k,p}^{(n)}$ we have
\begin{equation}
H^{p+q}_c(\mathscr{X}_{J,a}^{(n)}, \mathbb{C})  \cong H_{p-q}(X_{\gamma}(0), \mathbb{C}).
\end{equation}

\end{lemma}

\begin{proof}
Since $\mathscr{X}_{J,a}^{(n)}$ is a $p$-dimensional complex manifold, then by duality, we have 
$$H^{p+q}_c(\mathscr{X}_{J,a}^{(n)}, \mathbb{C}) \cong H_{p-q}(\mathscr{X}_{J,a}^{(n)}, \mathbb{C}).$$
 On the other hand, by Theorem \ref{mainthm1}, $\mathscr{X}_{J,a}^{(n)}$ is a locally trivial fibration on $X_{\gamma}(0)$ with fiber ${\mathbb C}^{|J_{\gamma}|(\ell_{J_{\gamma}}(a)+k)-s(a)-k}$ which is contractible. Hence, by the spectral sequence for (Serre) fibration, we obtain that
$H_{p-q}(\mathscr{X}_{J,a}^{(n)}, \mathbb{C})\cong H_{p-q}(X_{\gamma}(0), \mathbb{C}).$ The proof is complete
\end{proof}

We have the following result concerning cohomology of contact loci.

\begin{corollary}\label{corconstantsheaf}
Let $f\in\mathbb C[x_1,\dots,x_d]$ be nondegenerate, $n\in \mathbb N^*$ and $f(O)=0$.Then, there is a spectral sequence
\begin{equation}\label{spseqconstantsheaf}
E^{p,q}_1:=\bigoplus _{\gamma\in K}\Big(H_{p-q}(X_{\gamma}(1), \mathbb{C})^{|D_{\gamma,0,p}^{(n)}|}
\oplus\bigoplus_{k\geq 1}H_{p-q}(X_{\gamma}(0), \mathbb{C})^{|D_{\gamma,k,p}^{(n)}|}\Big)\Longrightarrow  H^{p+q}_c(\mathscr X_{n,O}(f), \mathbb{C}).
\end{equation}
\end{corollary}

\begin{proof} Apply Theorem \ref{thm5.6} for $\mathcal{F}$ to be the constant sheaf on $\mathscr X_{n,O}(f)$ associated to the field of complex numbers $\mathbb{C}$, since the inverse image of constant sheaf is a constant sheaf, the Corollary is a direct consequence of Theorem \ref{thm5.6} and the above lemmas.

\end{proof}

\begin{ack}
The first author is deeply grateful to Fran\c cois Loeser for introducing him to Problem \ref{pro1}. The authors thank Vietnam Institute for Advanced Study in Mathematics (VIASM) and Department of Mathematics - KU Leuven for warm hospitality during their visits. 
\end{ack}



\begin{thebibliography}{9999}
\bibitem{BFLN}
{N. Budur, J. Fernandez de Bobadilla, Q.T. L\^e, H.D. Nguyen}, {\it Cohomology of contact loci}, J. Differential Geom., to appear, arXiv:1911.08213.
	
\bibitem{BN}
{E. Bultot, J. Nicaise}, {\it Computing motivic zeta functions on log smooth models}, {Math. Z.} {\bf 295} (2020), 427--462.

\bibitem{DL1}
{J. Denef, F. Loeser}, {\it Motivic Igusa zeta functions}, {J. Algebraic Geom.} {\bf 7} (1998), 505--537.




\bibitem{G}
{G. Guibert}, {\it Espaces d'arcs et invariants d'Alexander}, {Comment. Math. Helv.} {\bf 77} (2002), 783--820.

\bibitem{GLM1}
{G. Guibert, F. Loeser, M. Merle}, {\it Iterated vanishing cycles, convolution, and a motivic analogue of the conjecture of Steenbrink}, {Duke Math. J.} {\bf 132} (2006), no. 3, 409--457.

\bibitem{GLM2}
{G. Guibert, F. Loeser, M. Merle}, {\it Nearby cycles and composition with a nondegenerate polynomial}, {International Mathematics Research Notices} {\bf 31} (2005), 1873--1888.

\bibitem{KS}
{M. Kontsevich, Y. Soibelman}, {\it Stability structures, motivic Donaldson-Thomas invariants and cluster transformations}, {arXiv: 0811.2435}.

\bibitem{LDT}
{D.T. L\^e}, {\it La monodromie n'a pas de points fixes}, {J. Fac. Sc. Univ. Tokyo}, sec. 1A, 1973.

\bibitem{Thuong1}
{Q.T. L\^e}, {\it On a conjecture of Kontsevich and Soibelman}, {Algebra and Number Theory} {\bf 6} (2012), no. 2, 389--404.

\bibitem{Thuong}
{Q.T. L\^e}, {\it Proofs of the integral identity conjecture over algebraically closed fields}, {Duke Math. J.}, {\bf 164} (2015), no. 1, 157--194.

\bibitem{Thuong2}
{Q.T. L\^e}, {\it A proof of the $\ell$-adic version of the integral identity conjecture for polynomials}, {Bull. Soc. Math. France} {\bf 147} (2019), 355--375.

\bibitem{LN}
{Q.T. L\^e, H.D. Nguyen}, {\it Equivariant motivic integration and proof of the integral identity conjecture for regular functions}, {Math. Ann.} {\bf 376} (2020), 1195--1223.

\bibitem{Mc}
{M. McLean}, {\it Floer cohomology, multiplicity, and the log canonical threshold}, {Geom. Topol.} {\bf 23} (2019), 957--1056.

\bibitem{NP}
{J. Nicaise and S. Payne}, {\it A tropical motivic Fubini theorem with applications to Donaldson-Thomas theory},  Duke Math. J. 168 (2019), no. 10, 1843--1886.

\bibitem{St2}
{J.H.M. Steenbrink}, {\it Motivic Milnor fibre for nondegenerate function germs on toric singularities}, Bridging algebra, geometry, and topology, 255--267, Springer Proc. Math. Stat., 96, Springer, Cham, 2014.
\end{thebibliography}
\end{document}